\documentclass[11pt]{amsart} 
\usepackage{amsthm,amsbsy,amsmath,amssymb,amscd,amsfonts,array,mathrsfs,verbatim,enumerate,xypic,enumitem}
\usepackage[all]{xy}
\xyoption{arc} 

\newtheorem{thm}{Theorem}
\newtheorem{prop}[thm]{Proposition}     
\newtheorem{lem}[thm]{Lemma}
\newtheorem{cor}[thm]{Corollary}

\theoremstyle{definition}

\newtheorem{defn}{Definition}

\newtheorem{example}{Example} 
\newtheorem{rem}{Remark}

\DeclareFontFamily{OT1}{rsfs}{}
\DeclareFontShape{OT1}{rsfs}{n}{it}{<-> rsfs10}{}
\DeclareMathAlphabet{\curly}{OT1}{rsfs}{n}{it}

\newcommand{\C}{{\bf C}} 
\newcommand{\D}{{\bf D}} 
\newcommand{\Ouv}{{\bf Ouv}} 
\newcommand{\RS}{{\bf RS}} 
\newcommand{\LRS}{{\bf LRS}} 
\newcommand{\Set}{{\bf Set}} 
\newcommand{\An}{{\bf Rings}} 
\newcommand{\Mon}{{\bf Mon}} 
\newcommand{\Alg}{{\bf Alg}} 
\newcommand{\LAn}{{\bf LAn}}   
\newcommand{\MS}{{\bf MS}}   
\newcommand{\LMS}{{\bf LMS}}   
\newcommand{\Sch}{{\bf Sch}} 
\newcommand{\PRS}{{\bf PRS}} 
\newcommand{\PS}{{\bf PS}} 
\newcommand{\Top}{{\bf Top}} 
\newcommand{\Sec}{{\bf Sec} \, } 

\newcommand{\m}{\mathfrak{m}}         
\newcommand{\n}{\mathfrak{n}}         
\newcommand{\p}{\mathfrak{p}}  

\newcommand{\ZZ}{\mathbb{Z}} 
\newcommand{\RR}{\mathbb{R}} 
\newcommand{\CC}{\mathbb{C}} 
\newcommand{\NN}{\mathbb{N}} 
\newcommand{\FF}{\mathbb{F}} 

\newcommand{\F}{\curly F}
\newcommand{\M}{\mathcal M}
\newcommand{\T}{\mathcal T}
\renewcommand{\O}{\mathcal O} 
\renewcommand{\P}{\mathcal P}

\renewcommand{\u}{\underline}

\newcommand{\ov}{\overline}
\newcommand{\into}{\hookrightarrow}
\newcommand{\be}{\begin{eqnarray*}}
\newcommand{\ee}{\end{eqnarray*}}
\newcommand{\bne}[1]{\begin{eqnarray} \label{#1} }
\newcommand{\ene}{\end{eqnarray}}
\newcommand{\xym}{\xymatrix}
\newcommand{\bp}{\begin{pmatrix}}
\newcommand{\ep}{\end{pmatrix}}
\newcommand{\slot}{ \hspace{0.05in} {\rm \_} \hspace{0.05in} } 

\newcommand{\loc}{\operatorname{loc}}
\newcommand{\Hom}{\operatorname{Hom}}

\newcommand{\U}{\operatorname{U}}

\newcommand{\Id}{\operatorname{Id}}
\newcommand{\Spec}{\operatorname{Spec}}

\begin{document}

\author{W.~D.~Gillam}
\address{Department of Mathematics, Brown University, Providence, RI 02912}
\email{wgillam@math.brown.edu}
\date{\today}
\title{Localization of ringed spaces}

\begin{abstract} Let $X$ be a ringed space together with the data $M$ of a set $M_x$ of prime ideals of $\O_{X,x}$ for each point $x \in X$.  We introduce the localization of $(X,M)$, which is a locally ringed space $Y$ and a map of ringed spaces $Y \to X$ enjoying a universal property similar to the localization of a ring at a prime ideal.  We use this to prove that the category of locally ringed spaces has all inverse limits, to compare them to the inverse limit in ringed spaces, and to construct a very general $\Spec$ functor.  We conclude with a discussion of relative schemes. \end{abstract}

\maketitle

\setcounter{section}{0}

\section{Introduction} Let $\Top$, $\LRS$, $\RS$, and $\Sch$ denote the categories of topological spaces, locally ringed spaces, ringed spaces, and schemes, respectively.  Consider maps of schemes $f_i : X_i \to Y$ ($i=1,2$) and their fibered product $X_1 \times_Y X_2$ as schemes.  Let $\u{X}$ denote the topological space underlying a scheme $X$.  There is a natural comparison map $$\eta : \u{X_1 \times_Y X_2} \to \u{X}_1 \times_{\u{Y}} \u{X}_2 $$ which is not generally an isomorphism, even if $X_1,X_2,Y$ are spectra of fields (e.g.\ if $Y=\Spec \RR$, $X_1=X_2=\Spec \CC$, the map $\eta$ is two points mapping to one point).  However, in some sense $\eta$ fails to be an isomorphism only to the extent to which it failed in the case of spectra of fields: According to [EGA I.3.4.7] the fiber $\eta^{-1}(x_1,x_2)$ over a point $(x_1,x_2) \in \u{X}_1 \times_{\u{Y}} \u{X}_2$ (with common image $y = f_1(x_1)=f_2(x_2)$) is naturally bijective with the set $$\u{ \Spec } \, k(x_1) \otimes_{k(y)} k(x_2)  . $$  In fact, one can show that this bijection is a homeomorphism when $\eta^{-1}(x_1,x_2)$ is given the topology it inherits from $X_1 \times_Y X_2$. One can even describe the sheaf of rings $\eta^{-1}(x_1,x_2)$ inherits from $X_1 \times_Y X_2$ as follows: Let \be S(x_1,x_2) & := & \{ z \in \Spec  \O_{X_1,x_1} \otimes_{\O_{Y,y}} \O_{X_2,x_2}  : z|\O_{X_i,x_i} = \m_{x_i} {\rm \; for \; } i=1,2 \} . \ee  Then ($\u{\Spec}$ of) the natural surjection $$  \O_{X_1,x_1} \otimes_{\O_{Y,y}} \O_{X_2,x_2} \to k(x_1) \otimes_{k(y)} k(x_2) $$ identifies $\u{\Spec} \, k(x_1) \otimes_{k(y)} k(x_2)$ with a closed subspace of $\u{\Spec} \, \O_{X_1,x_1} \otimes_{\O_{Y,y}} \O_{X_2,x_2}$ and $\O_{X_1 \times_Y X_2} | \eta^{-1}(x_1,x_2) $ naturally coincides, under the EGA isomorphism, to the restriction of the structure sheaf of $\Spec \O_{X_1,x_1} \otimes_{\O_{Y,y}} \O_{X_2,x_2}$ to the closed subspace $$\u{\Spec} \, k(x_1) \otimes_{k(y)} k(x_2) \subseteq \u{\Spec} \, \O_{X_1,x_1} \otimes_{\O_{Y,y}} \O_{X_2,x_2}.\footnote{There is no sense in which this sheaf of rings on $\u{\Spec} \, k(x_1) \otimes_{k(y)} k(x_2)$ is ``quasi-coherent".  It isn't even a module over the usual structure sheaf of $\Spec k(x_1) \otimes_{k(y)} k(x_2)$.}$$   It is perhaps less well-known that this entire discussion remains true for $\LRS$ morphisms $f_1,f_2$.

From the discussion above, we see that it is possible to describe $X_1 \times_Y X_2$, at least as a set, from the following data: \begin{enumerate} \item the ringed space fibered product $X_1 \times_Y^{\RS} X_2$ (which carries the data of the rings $\O_{X_1,x_1} \otimes_{\O_{Y,y}} \O_{X_2,x_2}$ as stalks of its structure sheaf) and \item the subsets $S(x_1,x_2) \subseteq \Spec \O_{X_1,x_1} \otimes_{\O_{Y,y}} \O_{X_2,x_2}$ \end{enumerate}  It turns out that one can actually recover $X_1 \times_Y X_2$ as a scheme solely from this data, as follows:  Given a pair $(X,M)$ consisting of a ringed space $X$ and a subset $M_x \subseteq \Spec \O_{X,x}$ for each $x \in X$, one can construct a locally ringed space $(X,M)^{\loc}$ with a map of ringed spaces $(X,M)^{\loc} \to X$.  In a special case, this construction coincides with M.~Hakim's spectrum of a ringed topos.  Performing this general construction to $$(X_1 \times_Y^{\RS} X_2, \{ S(x_1,x_2) \} ) $$ yields the comparison map $\eta$, and, in particular, the scheme $X_1 \times_Y X_2$.  A similar construction in fact yields all inverse limits in $\LRS$ (\S\ref{section:inverselimits}) and the comparison map to the inverse limit in $\RS$, and allows one to easily prove that a finite inverse limits of schemes, taken in $\LRS$, is a scheme (Theorem~\ref{thm:Schinverselimits}).  Using this description of the comparison map $\eta$ one can easily describe some circumstances under which it is an isomorphism (\S\ref{section:fiberedproducts}), and one can easily see, for example, that it is a \emph{localization morphism} (Definition~\ref{defn:localizationmorphism}), hence has zero cotangent complex.

The localization construction also allows us construct (\S\ref{section:relativespec}), for any $X \in \LRS$, a very general relative spec functor \be \Spec_X : (\O_X-\Alg)^{\rm op} & \to & \LRS / X \ee which coincides with the usual one when $X$ is a scheme and we restrict to quasi-coherent $\O_X$ algebras.  We can also construct (\S\ref{section:geometricrealization}) a ``good geometric realization" functor from M.~Hakim's stack of relative schemes over a locally ringed space $X$ to $\LRS / X$.\footnote{Hakim already constructed such a functor, but ours is different from hers.}  It should be emphasized at this point that there is essentially only \emph{one} construction, the localization of a ringed space of \S\ref{section:localization}, in this paper, and \emph{one} (fairly easy) theorem (Theorem~\ref{thm:localization}) about it; everything else follows formally from general nonsense.

Despite all these results about inverse limits, I stumbled upon this construction while studying \emph{direct limits}.  I was interested in comparing the quotient of, say, a finite \'etale groupoid in schemes, taken in sheaves on the \'etale site, with the same quotient taken in $\LRS$.  In order to compare these meaningfully, one must somehow put them in the same category.  An appealing way to do this is to prove that the (functor of points of the) $\LRS$ quotient is a sheaf on the \'etale site.  In fact, one can prove that for any $X \in \LRS$, the presheaf $$Y \mapsto \Hom_{\LRS}(Y,X) $$ is a sheaf on schemes in both the fppf and fpqc topologies.  Indeed, one can easily describe a topology on $\RS$, analogous to the fppf and fpqc topologies on schemes, and prove it is subcanonical.  To upgrade this to a subcanonical topology on $\LRS$ one is naturally confronted with the comparison of fibered products in $\LRS$ and $\RS$.  In particular, one is confronted with the question of whether $\eta$ is an epimorphism in the category of ringed spaces.  I do not know whether this is true for arbitrary $\LRS$ morphisms $f_1,f_2$, but in the case of schemes it is possible to prove a result along these lines which is sufficient to upgrade descent theorems for $\RS$ to descent theorems for $\Sch$. 

\noindent {\bf Acknowledgements.}  This research was partially supported by an NSF Postdoctoral Fellowship.

\section{Localization} \label{section:mainconstruction}  We will begin the localization construction after making a few definitions.  

\begin{defn} \label{defn:localizationmorphism} A morphism $f : A \to B$ of sheaves of rings on a space $X$ is called a \emph{localization morphism}\footnote{See [Ill II.2.3.2] and the reference therein.} iff there is a multiplicative subsheaf $S \subseteq A$ so that $f$ is isomorphic to the localization $A \to S^{-1}A$ of $A$ at $S$.\footnote{This condition can be checked in stalks.}  A morphism of ringed spaces $f : X \to Y$ is called a localization morphism iff $f^{\sharp} : f^{-1} \O_Y \to \O_X$ is a localization morphism.  \end{defn}

A localization morphism $A \to B$ in $\An(X)$ is both flat and an epimorphism in $\An(X)$.\footnote{Both of these conditions can be checked at stalks.}  In particular, the cotangent complex (hence also the sheaf of K\"ahler differentials) of a localization morphism is zero [Ill II.2.3.2].  The basic example is: For any affine scheme $X = \Spec A$, $\underline{A}_X \to \O_X$ is a localization morphism.

\begin{defn} \label{defn:Spec} Let $A$ be a ring, $S \subseteq \Spec A$ any subset.  We write $\Spec_A S$ for the locally ringed space whose underlying topological space is $S$ with the topology it inherits from $\Spec A$ and whose sheaf of rings is the inverse image of the structure sheaf of $\Spec A$. \end{defn}

If $A$ is clear from context, we drop the subscript and simply write $\Spec S$.  There is one possible point of confusion here: If $I \subseteq A$ is an ideal, and we think of $\Spec A/I$ as a subset of $\Spec A$, then \be \Spec_A (\Spec A/I) & \neq & \Spec A/I  \ee (though they have the same topological space).

\subsection{Prime systems} 

\begin{defn} \label{defn:PRS} Let $X=(X,\O_X)$ be a ringed space.  A \emph{prime system} $M$ on $X$ is a map $x \mapsto M_x$ assigning a subset $M_x \subseteq \Spec \O_{X,x}$ to each point $x \in X$.  For prime systems $M,N$ on $X$ we write $M \subseteq N$ to mean $M_x \subseteq N_x$ for all $x \in X$.  Prime systems on $X$ form a category $\PS(X)$ where there is a unique morphism from $M$ to $N$ iff $M \subseteq N$.  The \emph{intersection} $\cap_i M_i$ of prime systems $M_i \in \PS(X)$ is defined by \be ( \cap_i M_i)_x & := & \cap_i (M_i)_x. \ee A \emph{primed ringed space} $(X,M)$ is a ringed space $X$ equipped with a prime system $M$.  Prime ringed spaces form a category $\PRS$ where a morphism $f : (X,M) \to (Y,N)$ is a morphism of ringed spaces $f$ satisfying $$ (\Spec f_x)(M_x) \subseteq N_{f(x)} $$ for every $x \in X$. \end{defn}

The inverse limit of a functor $i \mapsto M_i$ to $\PS(X)$ is clearly given by $\cap_i M_i$.

\begin{rem} \label{rem:inverseimage} Suppose $(Y,N) \in \PRS$ and $f : X \to Y$ is an $\RS$ morphism.  The \emph{inverse image}  $f^* N$ is the prime system on $X$ defined by \be (f^*N)_x & := & (\Spec f_x)^{-1}(N_{f(x)}) \\ & = & \{ \p \in \Spec \O_{X,x} : f_x^{-1}(\p) \in N_{f(x)} \}. \ee  Formation of inverse image prime systems enjoys the expected naturality in $f$: $g^*(f^*M) = (fg)^*M$.  We can alternatively define a $\PRS$ morphism $f : (X,M) \to (Y,N)$ to be an $\RS$ morphism $f : X \to Y$ such that $M \subseteq f^* N$ (i.e.\ together with a $\PS(X)$ morphism $M \to f^* N$). \end{rem}

For $X \in \LRS$, the \emph{local prime system} $\M_X$ on $X$ is defined by $\M_{X,x} :=  \{ \m_x \}$.  If $Y$ is another locally ringed space, then a morphism $f : X \to Y$ in $\RS$ defines a morphism of primed ringed spaces $f : (X,\M_X) \to (Y,\M_Y)$ iff $f$ is a morphism in $\LRS$, so we have a fully faithful functor \bne{M}  \M : \LRS & \to & \PRS \\ \nonumber X & \mapsto & (X, \M_X), \ene and we may regard $\LRS$ as a full subcategory of $\PRS$. 

At the ``opposite extreme" we also have, for any $X \in \RS$, the \emph{terminal prime system} $\T_X$ defined by $\T_{X,x} := \Spec \O_{X,x}$ (i.e.\ the terminal object in $\PS(X)$).  For $(Y,M) \in \PRS$, we clearly have \be \Hom_{\PRS}((Y,M), (X,\T_X)) & = & \Hom_{\RS}(Y,X), \ee  so the functor \bne{T} \T : \RS & \to & \PRS \\ \nonumber X & \mapsto & (X,\T_X) \ene is right adjoint to the forgetful functor $\PRS \to \RS$ given by $(X,M) \mapsto X$.

\subsection{Localization} \label{section:localization} Now we begin the main construction of this section.  Let $(X,M)$ be a primed ringed space.  We now construct a locally ringed space $(X,M)^{\loc}$ (written $X^{\loc}$ if $M$ is clear from context), and a $\PRS$ morphism $\pi : (X^{\loc}, \M_{X^{\loc}}) \to (X,M)$ called the \emph{localization of} $X$ \emph{at} $M$. 

\begin{defn} Let $X$ be a topological space, $\F$ a sheaf on $X$.  The category $\Sec \F$ of \emph{local sections} of $\F$ is the category whose objects are pairs $(U,s)$ where $U$ is an open subset of $X$ and $s \in \F(U)$, and where there is a unique morphism $(U,s) \to (V,t)$ iff $U \subseteq V$ and $t|_U = s$. \end{defn}

As a set, the topological space $X^{\loc}$ will be the set of pairs $(x,z)$, where $x \in X$ and $z \in M_x$.  Let $\P(X^{\loc})$ denote the category of subsets of $X^{\loc}$ whose morphisms are inclusions.  For $(U,s) \in \Sec \O_X$, set \be \U(U,s) & := & \{ (x,z) \in X^{\loc} : x \in U, s_x \notin z \} . \ee  This defines a functor \be \U : \Sec \O_X & \to & \P(X^{\loc}) \ee satisfying: \be \U(U,s) \cap \U(V,t) & = & \U(U \cap V, s|_{U \cap V} t|_{U \cap V}) \\ \U(U,s^n) & = & \U(U,s) \quad \quad \quad \quad (n \in \ZZ_{>0}). \ee  The first formula implies that $\U(\Sec \O_X) \subseteq \P(X^{\loc})$ is a basis for a topology on $X^{\loc}$ where a basic open neighborhood of $(x,z)$ is a set $\U(U,s)$ where $x \in U$, $s_x \notin z$.  We always consider $X^{\loc}$ with this topology.  The map \be \pi : X^{\loc} & \to & X \\ (x,z) & \mapsto & x \ee is continuous because $\pi^{-1}(U) = \U(U,1)$.

We construct a sheaf of rings $\O_{X^{\loc}}$ on $X^{\loc}$ as follows.  For an open subset $V \subseteq X^{\loc}$, we let $\O_{X^{\loc}}(V)$ be the set of $$ s = (s(x,z)) \in \prod_{(x,z) \in V} (\O_{X,x})_z $$ satisfying the \emph{local consistency condition}: For every $(x,z) \in V$, there is a basic open neighborhood $\U(U,t)$ of $(x,z)$ contained in $V$ and a section $$ \frac{a}{t^n} \in \O_X(U)_t $$ such that, for every $(x',z') \in \U(U,t)$, we have \be s(x',z') & = & \frac{a_{x'}}{t_{x'}^n} \in (\O_{X,x'})_{z'}. \ee  (Of course, one can always take $n=1$ since $\U(U,t)=\U(U,t^n)$.)  The set $\O_{X^{\loc}}(V)$ becomes a ring under coordinatewise addition and multiplication, and the obvious restriction maps make $\O_{X^{\loc}}$ a sheaf of rings on $X^{\loc}$.  There is a natural isomorphism \be \O_{X^{\loc},(x,z)} & = & (\O_{X,x})_z \ee taking the germ of $s=(s(x,z)) \in \O_{X^{\loc}}(U)$ in the stalk $\O_{X^{\loc},(x,z)}$ to $s(x,z) \in (\O_{X,x})_z$.  This map is injective because of the local consistency condition and surjective because, given any $a/b \in (\O_{X,x})_z$, we can lift $a,b$ to $\ov{a},\ov{b} \in \O_X(U)$ on some neighborhood $U$ of $x$ and define $s \in \O_{X^{\loc}}(\U(U,\ov{b}))$ by letting $s(x',z') := \ov{a}_{x'} / \ov{b}_{x'} \in (\O_{X,x'})_{z'}.$  This $s$ manifestly satisfies the local consistency condition and has $s(x,z) = a/b$.  In particular, $X^{\loc}$, with this sheaf of rings, is a locally ringed space.

To lift $\pi$ to a map of ringed spaces $\pi : X^{\loc} \to X$ we use the tautological map $$\pi^{\flat} :\O_X \to \pi_* \O_{X^{\loc}}$$ of sheaves of rings on $X$ defined on an open set $U \subseteq X$ by \be \pi^{\flat}(U) : \O_X(U) & \to & (\pi_* \O_{X^{\loc}})(U) = \O_{X^{\loc}}(\U(U,1)) \\ s & \mapsto & ( (s_x)_z ) . \ee  It is clear that the induced map on stalks \be \pi_{x,z} : \O_{X,x} & \to & \O_{X^{\loc},(x,z)} = (\O_{X,x})_z  \ee is the natural localization map, so $\pi_{x,z}^{-1}(\m_z) = z \in M_x$ and hence $\pi$ defines a $\PRS$ morphism $\pi : (X^{\loc}, \M_{X^{\loc}}) \to (X,M)$.

\begin{rem} \label{rem:localization} It would have been enough to construct the localization $(X,\T_X)^{\loc}$ at the terminal prime system.  Then to construct the localization $(X,M)^{\loc}$ at any other prime system, we just note that $(X,M)^{\loc}$ is clearly a subset of $(X,\T_X)^{\loc}$, and we give it the topology and sheaf of rings it inherits from this inclusion.  The construction of $(X,\T_X)^{\loc}$ is ``classical."  Indeed, M.~Hakim \cite{Hak} describes a construction of $(X,\T_X)^{\loc}$ that makes sense for any ringed topos $X$ (she calls it the \emph{spectrum} of the ringed topos [Hak IV.1]), and attributes the analogous construction for ringed spaces to C.~Chevalley [Hak IV.2].  Perhaps the main idea of this work is to define ``prime systems," and to demonstate their ubiquity.  The additional flexibility afforded by non-terminal prime systems is indispensible in the applications of \S\ref{section:applications}.  It is not clear to me whether this setup generalizes to ringed topoi.  \end{rem}

We sum up some basic properties of the localization map $\pi$ below.

\begin{prop} \label{prop:localization} Let $(X,M)$ be a primed ringed space with localization $\pi : X^{\loc} \to X$.  For $x \in X$, the fiber $\pi^{-1}(x)$ is naturally isomorphic in $\LRS$ to $\Spec M_x$ (Definition~\ref{defn:Spec}).\footnote{By ``fiber" here we mean $\pi^{-1}(x) := X^{\loc} \times^{\RS}_X (\{ x \}, \O_{X,x})$, which is just the set theoretic preimage $\pi^{-1}(x) \subseteq X^{\loc}$ with the topology and sheaf of rings it inherits from $X^{\loc}$.  This differs from another common usage of ``fiber" to mean $X^{\loc} \times^{\RS}_X (\{ x \}, k(x) )$. }  Under this identification, the stalk of $\pi$ at $z \in M_x$ is identified with the localization of $\O_{X,x}$ at $z$, hence $\pi$ is a localization morphism (Definition~\ref{defn:localizationmorphism}).  \end{prop}

\begin{proof} With the exception of the fiber description, everything in the proposition was noted during the construction of the localization.  Clearly there is a natural bijection of sets $M_x = \pi^{-1}(x)$ taking $z \in M_x$ to $(x,z) \in \pi^{-1}(x)$.  We first show that the topology inherited from $X^{\loc}$ coincides with the one inherited from $\Spec \O_{X,x}$.  By definition of the topology on $X^{\loc}$, a basic open neighborhood of $z \in M_x$ is a set of the form \be \U(U,s) \cap M_x & = & \{ z' \in M_x : s_x \notin z' \}, \ee where $U$ is a neighborhood of $x$ in $X$ and $s \in \O_X(U)$ satisfies $s_x \notin z$.  Clearly this set depends only on the stalk of $s_x \in \O_{X,x}$ of $s$ at $x$, and any element $t \in \O_{X,x}$ lifts to a section $\ov{t} \in \O_X(U)$ on some neighborhood of $X$, so the basic neighborhoods of $z \in M_x$ are the sets of the form $$ \{ z' \in M_x : t \notin z' \}$$ where $t_x \notin z$.  But for the same set of $t$, the sets \be D(t) & := & \{ \p \in \Spec \O_{X,x} : t \notin \p \} \ee form a basis for neighborhoods of $z$ in $\Spec \O_{X,x}$ so the result is clear. 

We next show that the sheaf of rings on $M_x$ inherited from $X^{\loc}$ is the same as the one inherited from $\Spec \O_{X,x}$.  Given $f \in \O_{X,x}$, a section of $\O_{X^{\loc}}|M_x$ over the basic open set $M_x \cap D(f)$ is an element $$s = (s(z)) \in \prod_{z \in M_x \cap D(f)} (\O_{X,x})_z $$ satisfying the local consistency condition: For all $z \in M_x \cap D(f)$, there is a basic open neighborhood $\U(U,t)$ of $(x,z)$ in $X^{\loc}$ and an element $a/(t^n) \in \O_X(U)_t$ such that, for all $z' \in M_x \cap D(f) \cap \U(U,t)$, we have $s(z') \in a_{z'} / (t^n_{z'})$.  Note that \be M_x \cap D(f) \cap \U(U,t) & = & M_x \cap D(ft_x) \ee and $a_x / (t_x^n) \in \O_{\Spec \O_{X,x}}(D(ft_x))$.  The sets $D(ft_x) \cap M_x$ cover $M_x \cap D(f) \subseteq \Spec \O_{X,x}$, and we have a ``global formula" $s$ showing that the stalks of the various $a_x / (t_x^n)$ agree at any $z \in M_x \cap D(f)$, so they glue to yield an element $g(s) \in \O_{\Spec \O_{X,x}}(M_x \cap D(f))$ with $g(s)_z = s(z)$.  We can define a morphism of sheaves on $M_x$ by defining it on basic opens, so this defines a morphism of sheaves $g : \O_{X^{\loc}}|M_x \to \O_{\Spec \O_{X,x}}|M_x$ which is easily seen to be an isomorphism on stalks.  \end{proof}

\begin{rem} \label{rem:open}  Suppose $(X,M) \in \PRS$ and $U \subseteq X$ is an open subspace of $X$.  Then it is clear from the construction of $\pi : (X,M)^{\loc} \to X$ that $\pi^{-1}(U) = (U,\O_X|U,M|U)^{\loc}$. \end{rem}

The following theorem describes the universal property of localization.

\begin{thm} \label{thm:localization} Let $f: (X,M) \to (Y,N)$ be a morphism in $\PRS$.  Then there is a unique morphism $\ov{f} : (X,M)^{\loc} \to (Y,N)^{\loc}$ in $\LRS$ making the diagram \bne{dia} & \xym{ (X,M)^{\loc} \ar[d]_{\pi} \ar[r]^{\ov{f}} & (Y,N)^{\loc} \ar[d]^{\pi} \\ X \ar[r]^f & Y } \ene commute in $\RS$.  Localization defines a functor \be \PRS & \to & \LRS \\ (X,M) & \mapsto & (X,M)^{\loc} \\ f: (X,M) \to (Y,N) & \mapsto & \ov{f} : (X,M)^{\loc} \to (Y,N)^{\loc} \ee retracting the inclusion functor $\M : \LRS \to \PRS$ and right adjoint to it:  For any $Y \in \LRS$, there is a natural bijection \be \Hom_{\LRS}(Y, (X,M)^{\loc}) & = & \Hom_{\PRS}((Y,\M_Y), (X,M)). \ee  \end{thm}

\begin{proof}  We first establish the \emph{existence} of such a morphism $\ov{f}$.  The fact that $f$ is a morphism of primed ringed spaces means that we have a function \be  M_x & \to & N_{f(x)} \\ z & \mapsto & f_x^{-1}(z) \ee for each $x \in X$, so we can complete the diagram of topological spaces $$ \xym{ X^{\loc} \ar[d]^{\pi} \ar[r]^{\ov{f}} & (Y,N)^{\loc} \ar[d]^{\pi}  \\ X \ar[r]^{f}  & Y } $$ (at least on the level of sets) by setting \be \ov{f}(x,z) & := & (f(x), f_x^{-1}(z)) \in Y^{\loc}. \ee  To see that $\ov{f}$ is continuous it is enough to check that the preimage $\ov{f}^{-1}(\U(U,s))$ is open in $X^{\loc}$ for each basic open subset $\U(U,s)$ of $Y^{\loc}$.  But it is clear from the definitions that \be \ov{f}^{-1} \U(U,s) & = & \U( f^{-1}(U), f^{\sharp} f^{-1}(s) ) \ee (note $(f^{\sharp} f^{-1}(s))_x = f_x(s_{f(x)})$).  

Now we want to define a map $\ov{f}^{\sharp} : \ov{f}^{-1} \O_{X^{\loc}} \to \O_Y$ of sheaves of rings on $Y$ (with ``local stalks") making the diagram $$ \xym@C+20pt{ \O_{X^{\loc}} & \ar@{.>}[l]_-{\ov{f}^{\sharp}} \ov{f}^{-1} \O_{Y^{\loc}} \\ \pi^{-1} \O_X \ar[u] & \ar[l]_-{\pi^{-1} f^{\sharp}} \pi^{-1} f^{-1} \O_Y \ar[u] } $$ commute in $\An(Y^{\loc})$.  The stalk of this diagram at $(x,z) \in X^{\loc}$ is a diagram $$ \xym@C+20pt{ (\O_{X,x})_z & \ar@{.>}[l]_-{\ov{f}_{(x,z)}} (\O_{Y,f(x)})_{f_x^{-1}z} \\ \O_{X,x} \ar[u]^{\pi_{(x,z)}} & \ar[l]^{f_x}  \O_{Y,f(x)} \ar[u]_{\pi_{(f(x),f_x^{-1}(z))}} } $$ in $\An$ where the vertical arrows are the natural localization maps; these are epimorphisms, and the universal property of localization ensures that there is a unique local morphism of local rings $\ov{f}_{(x,z)}$ completing this diagram.  We now want to show that there is actually a (necessarily unique) map $\ov{f}^{\sharp} : \ov{f}^{-1} \O_{X^{\loc}} \to \O_Y$ of sheaves of rings on $X^{\loc}$ whose stalk at $(x,z)$ is the map $\ov{f}_{(x,z)}$.  By the universal property of sheafification, we can work with the presheaf inverse image $\ov{f}^{-1}_{\rm pre} \O_{X^{\loc}}$ instead.  A section $[V,s]$ of this presheaf over an open subset $W \subseteq X^{\loc}$ is represented by a pair $(V,s)$ where $V \subseteq Y^{\loc}$ is an open subset of $Y^{\loc}$ containing $\ov{f}(W)$ and $$ s = (s(y,z)) \in \O_{Y^{\loc}}(V) \subseteq \prod_{(y,z) \in V} (\O_{Y,y})_z . $$  I claim that we can define a section $f_{\rm pre}^{\sharp}[V,s] \in \O_{X^{\loc}}(W) $ by the formula \be f_{\rm pre}^{\sharp}[V,s](x,z)  :=  s(f(x),f_x^{-1}(z))  . \ee  It is clear that this element is independent of replacing $V$ with a smaller neighborhood of $f[W]$ and restricting $s$, but we still must check that $$ f_{\rm pre}[V,s] \in \prod_{(x,z) \in W} (\O_{X,x})_z$$  satisfies the local consistency condition.  Suppose $$ \frac{a}{t^n} \in \O_X(U)_t$$ witnesses local consistency for $s \in \O_{Y^{\loc}}(V)$ on a basic open subset $\U(U,t) \subseteq V$.  Then it is straightforward to check that the restriction of $$ \frac{ f^{\sharp}(f^{-1}(a)) }{ f^{\sharp}(f^{-1}t^n)} \in \O_Y( \ov{f}^{-1}( \U(U,t) )) $$ to $\ov{f}^{-1} \U(U,t) \cap W $ witnesses local consistency of $f_{\rm pre}^{\sharp}[V,s]$ on \be \ov{f}^{-1}(\U(U,t)) \cap W & = & \U(f^{-1}(U), f^\sharp f^{-1} t) \cap W . \ee  It is clear that our formula for $f^{\sharp}_{\rm pre}[V,s]$ respects restrictions and has the desired stalks and commutativity, so its sheafification provides the desired map of sheaves of rings.

This completes the construction of $\ov{f} : X^{\loc} \to Y^{\loc}$ in $\LRS$ making \eqref{dia} commute in $\RS$.  We now establish the uniqueness of $\ov{f}$.  Suppose $\ov{f}' : X^{\loc} \to Y^{\loc}$ is a morphism in $\LRS$ that also makes \eqref{dia} commute in $\RS$.   We first prove that $\ov{f} = \ov{f}'$ on the level of topological spaces.  For $x \in X$ the commutativity of \eqref{dia} ensures that $\ov{f}'(x,z) = (f(x),z')$ for some $z ' \in N_{f(x)} \subseteq \Spec \O_{Y,f(x)},$ so it remains only to show that $z' = f_x^{-1}(z).$  The commutativity of \eqref{dia} on the level of stalks at $(x,z) \in X^{\loc}$ gives a commutative diagram of rings $$ \xym{ (\O_{X,x})_z & \ar[l]_{\ov{f}'_{(x,z)}} (\O_{Y,f(x)})_{z'} \\ \O_{X,x} \ar[u]^{\pi_{(x,z)}} & \ar[l]^{f_x}  \O_{Y,f(x)} \ar[u]_{\pi_{(f(x),z')}} } $$ where the vertical arrows are the natural localization maps.  From the commutativity of this diagram and the fact that $(\ov{f}')_{(x,z)}^{-1}(\m_z) = \m_{z'}$ (because $\ov{f}'_{(x,z)}$ is local) we find \be z' & = & \pi_{(f(x),z')}^{-1}(\m_{z'}) \\ & = & \pi_{(f(x),z')}^{-1} (\ov{f}')_{(x,z)}^{-1}(\m_z) \\ & = & f_x^{-1} \pi_x^{-1}(\m_z)  \\ & = & f_x^{-1}(z) \ee as desired. This proves that $\ov{f} = \ov{f}'$ on topological spaces, and we already argued the uniqueness of $\ov{f}^{\sharp}$ (which can be checked on stalks) during its construction.    

The last statements of the theorem follow easily once we prove that the localization morphism $\pi : (X,\M_X)^{\loc} \to X$ is an isomorphism for any $X \in \LRS$.  On the level of topological spaces, it is clear that $\pi$ is a continuous bijection, so to prove it is an isomorphism we just need to prove it is open.  To prove this, it is enough to prove that for any $(U,s) \in \Sec \O_X$, the image of the basic open set $\U(U,s)$ under $\pi$ is open in $X$.  Indeed, \be \pi( \U(U,s) ) & = & \{ x \in U : s_x \notin \m_x \} \\ & = & \{ x \in U : s_x \in \O_{X,x}^* \} \ee is open in $U$, hence in $X$, because invertibility at the stalk implies invertibility on a neighborhood.  To prove that $\pi$ is an isomorpism of locally ringed spaces, it remains only to prove that $\pi^{\sharp} : \O_X \to \O_{X^{\loc}}$ is an isomorphism of sheaves of rings on $X=X^{\loc}$.   Indeed,  Proposition~\ref{prop:localization} says the stalk of $\pi^{\sharp}$ at $(x, \m_x) \in X^{\loc}$ is the localization of the local ring $\O_{X,x}$ at its unique maximal ideal, which is an isomorphism in $\LAn$. \end{proof}

\begin{lem} \label{lem:important} Let $A \in \An$ be a ring, $(X,\O_X) := \Spec A$, and let $*$ be the punctual space.  Define a prime system $N$ on $(X,\underline{A}_X)$ by $$N_x := \{ x \} \subseteq \Spec \underline{A}_{X,x} = \Spec A = X.$$  Let $a : (X,\O_X) \to (X,\underline{A}_X)$ be the natural $\RS$ morphism.  Then $\M_{(X,\O_X)} = a^* N$ and the natural $\PRS$ morphisms $$(X,\O_X,\M_{(X,\O_X)}) \to (X,\underline{A}_X,N) \to (*,A,\Spec A) = (*,A,\T_{(*,A)}) $$ yield natural isomorphisms $$ (X,\O_X) = (X,\O_X,\M_{(X,\O_X)})^{\loc} = (X,\underline{A}_X,N)^{\loc} = (*,A,\Spec A)^{\loc} $$ in $\LRS$. \end{lem}

\begin{proof} Note that the stalk $a_x : \underline{A}_{X,x} \to \O_{X,x}$ of $a$ at $x \in X$ is the localization map $A \to A_x$, and, by definition, $(a^*N)_x$ is the set prime ideals $z$ of $A_x$ pulling back to $x \subseteq A$ under $a_x : A \to A_x$.  The only such prime ideal is the maximal ideal $\m_x \subseteq A_x$, so $(a^*N)_x = \{ \m_x \} = \M_{(X,\O_X),x}$.  

Next, it is clear from the description of the localization of a $\PRS$ morphism that the localizations of the morphisms in question are bijective on the level of sets.  Indeed, the bijections are given by $$ x \leftrightarrow (x,\m_x) \leftrightarrow (x,x) \leftrightarrow (*,x),$$ so to prove that they are continuous, we just need to prove that they have the same topology.  Indeed, we will show that they all have the usual (Zariski) topology on $X=\Spec A$.  This is clear for $(X,\O_X,\M_{(X,\O_X)})$ because localization retracts $\M$ (Theorem~\ref{thm:localization}), so $(X,\O_X,\M_{(X,\O_X)})^{\loc} = (X,\O_X)$, and it is clear for $(*,A,\Spec A)$ because of the description of the fibers of localization in Proposition~\ref{prop:localization}.  For $(X,\underline{A}_X,N)$, we note that the sets $\U(U,s)$, as $U$ ranges over \emph{connected} open subsets of $X$ (or any other family of basic opens for that matter), form a basis for the topology on $(X,\underline{A}_X,N)^{\loc}$.  Since $U$ is connected, $s \in \underline{A}_X(U) = A$, and $\U(U,s)$ is identified with the usual basic open subset $D(s) \subseteq X$ under the bijections above.  This proves that the $\LRS$ morphisms in question are isomorphisms on the level of spaces, so it remains only to prove that they are isomorphisms on the level of sheaves of rings, which we can check on stalks using the description of the stalks of a localization in Proposition~\ref{prop:localization}. \end{proof}

\begin{rem} \label{rem:notlocal} If $X \in \LRS$, and $M$ is a prime system on $X$, the map $\pi : X^{\loc} \to X$ is not generally a morphism in $\LRS$, even though $X,X^{\loc} \in \LRS$.  For example, if $X$ is a point whose ``sheaf" of rings is a local ring $(A,\m)$, and $M = \{ \p \}$ for some $\p \neq \m$, then $X^{\loc}$ is a point with the ``sheaf" of rings $A_{\p}$, and the ``stalk" of $\pi^{\sharp}$ is the localization map $l : A \to A_{\p}$.  Even though $A, A_{\p}$ are local, this is \emph{not} a local morphism because $l^{-1}( \p A_{\p}) = \p \neq \m$. \end{rem}

\section{Applications} \label{section:applications} In this section we give some applications of localization of ringed spaces.

\subsection{Inverse limits}  \label{section:inverselimits} We first prove that $\LRS$ has all inverse limits.

\begin{thm} \label{thm:inverselimits} The category $\PRS$ has all inverse limits, and both the localization functor $\PRS \to \LRS$ and the forgetful functor $\PRS \to \RS$ preserve them. \end{thm}

\begin{proof} Suppose $i \mapsto (X_i,M_i)$ is an inverse limit system in $\PRS$.  Let $X$ be the inverse limit of $i \mapsto X_i$ in $\Top$ and let $\pi_i : X \to X_i$ be the projection.  Let $\O_X$ be the direct limit of $i \mapsto \pi_i^{-1} \O_{X_i}$ in $\An(X)$ and let $\pi_i^{\sharp} : \pi_i^{-1} \O_{X_i} \to \O_X$ be the structure map to the direct limit, so we may regard $X=(X,\O_X)$ as a ringed space and $\pi_i$ as a morphism of ringed spaces $X \to X_i$.  It immediate from the definition of a morphism in $\RS$ that $X$ is the inverse limit of $i \mapsto X_i$ in $\RS$.  Let $M$ be the prime system on $X$ given by the inverse limit (intersection) of the $\pi_i^*M_i$.  Then it is clear from the definition of a morphism in $\PRS$ that $(X,M)$ is the inverse limit of $i \mapsto (X_i,M_i)$, but we will spell out the details for the sake of concreteness and future use.  

Given a point $x = (x_i) \in X$, we have defined $M_x$ to be the set of $z \in \Spec \O_{X,x}$ such that $\pi_{i,x}^{-1}(z) \in M_{x_i} \subseteq \Spec \O_{X_i,x_i}$ for every $i$, so that $\pi_i$ defines a $\PRS$ morphism $\pi_i : (X,M) \to (X_i,M_i)$.  To see that $(X,M)$ is the direct limit of $i \mapsto (X_i,M_i)$ suppose $f_i : (Y,N) \to (X_i,M_i)$ are morphisms defining a natural transformation from the constant functor $i \mapsto (Y,N)$ to $i \mapsto (X_i,M_i)$.  We want to show that there is a unique $\PRS$ morphism $f : (Y,N) \to (X,M)$ with $\pi_i f = f_i$ for all $i$.  Since $X$ is the inverse limit of $i \mapsto X_i$ in $\RS$, we know that there is a unique map of ringed spaces $f : Y \to X$ with $\pi_i f = f_i$ for all $i$, so it suffices to show that this $f$ is a $\PRS$ morphism.  Let $y \in Y$, $z \in N_y$.  We must show $f_y^{-1}(z) \in M_{f(x)}$.  By definition of $M$, we must show $(\pi_i)_{f(x)}^{-1}(f_y^{-1}(z)) \in M_{\pi_i(f(x))} = M_{f_i(y)}$ for every $i$.  But $\pi_i f = f_i$ implies $f_y (\pi_i)_{f(x)} = (f_i)_y$, so $(\pi_i)_{f(x)}^{-1}(f_y^{-1}(z)) = (f_i)_y^{-1}(z) $ is in $M_{f_i(y)}$ because $f_i$ is a $\PRS$ morphism.

The fact that the localization functor preserves inverse limits follows formally from the adjointness in Theorem~\ref{thm:localization}. \end{proof}

\begin{cor} \label{cor:inverselimits} The category $\LRS$ has all inverse limits. \end{cor}

\begin{proof} Suppose $i \mapsto X_i$ is an inverse limit system in $\LRS$.  Composing with $\M$ yields an inverse limit system $i \mapsto (X_i,\M_{X_i})$ in $\PRS$.  By the theorem, the localization $(X,M)^{\loc}$ of the inverse limit $(X,M)$ of $i \mapsto (X_i,\M_{X_i})$ is the inverse limit of $i \mapsto (X_i,\M_{X_i})^{\loc}$ in $\LRS$.  But localization retracts $\M$  (Theorem~\ref{thm:localization}) so $i \mapsto (X_i,\M_{X_i})^{\loc}$ is our original inverse limit system $i \mapsto X_i$. \end{proof} 

We can also obtain the following result of C.~Chevalley mentioned in [Hak IV.2.4].

\begin{cor} \label{cor:Chevalley} The functor \be \RS & \to & \LRS \\ X & \mapsto & (X,\T_X)^{\loc} \ee is right adjoint to the inclusion $\LRS \into \RS$. \end{cor}

\begin{proof} This is immediate from the adjointness property of localization in Theorem~\ref{thm:localization} and the adjointness property of the functor $\T$: For $Y \in \LRS$ we have \be \Hom_{\LRS}(Y,(X,\T_X)^{\loc}) & = & \Hom_{\PRS}((Y,\M_Y), (X,\T_X)) \\ & = & \Hom_{\RS}(Y,X). \ee \end{proof}

Our next task is to compare inverse limits in $\Sch$ to those in $\LRS$.  Let $* \in \Top$ be ``the" punctual space (terminal object), so $\An(*) = \An$.  The functor \be \An & \to & \RS \\ A & \mapsto & (*,A) \ee is clearly left adjoint to \be \Gamma : \RS^{\rm op} & \to & \An \\ X & \mapsto & \Gamma(X,\O_X) . \ee By Lemma~\ref{lem:important} (or Proposition~\ref{prop:localization}) we have \be \T(*,A)^{\loc} & := & (*,A,\Spec A)^{\loc} \\ & = & \Spec A .\ee  Theorem~\ref{thm:localization} yields an easy proof of the following result, which can be found in the Errata for [EGA I.1.8] printed at the end of [EGA II].

\begin{prop} \label{prop:affine} For $A \in \An$, $X \in \LRS$, the natural map \be \Hom_{\LRS}(X,\Spec A) & \to & \Hom_{\An}(A, \Gamma(X,\O_X)) \ee is bijective, so $\Spec : \An \to \LRS$ is left adjoint to $\Gamma : \LRS^{\rm op} \to \An$.   \end{prop}

\begin{proof} This is a completely formal consequence of various adjunctions: \be \Hom_{\LRS}(X,\Spec A) & = & \Hom_{\LRS}(X,\T(*,A)^{\loc}) \\ & = & \Hom_{\PRS}((X,\M_X),\T(*,A)) \\ & = & \Hom_{\RS}(X, (*,A)) \\ & = & \Hom_{\An}(A,\Gamma(X,\O_X)). \ee  \end{proof}

\begin{thm} \label{thm:Schinverselimits} The category $\Sch$ has all finite inverse limits, and the inclusion $\Sch \to \LRS$ preserves them. \end{thm}

\begin{proof} It is equivalent to show that, for a finite inverse limit system $i \mapsto X_i$ in $\Sch$, the inverse limit $X$ in $\LRS$ is a scheme.  It suffices to treat the case of (finite) products and equalizers.  For products, suppose $\{ X_i \}$ is a finite set of schemes and $X = \prod_i X_i$ is their product in $\LRS$.  We want to show $X$ is a scheme.  Let $\ov{x}$ be a point of $X$, and let $x = (x_i) \in \prod_i^{\RS} X_i$ be its image in the ringed space product.  Let $U_i = \Spec A_i$ be an open affine neighborhood of $x_i$ in $X_i$.  As we saw above, the map $X \to \prod_i^{\RS} X_i$ is a localization and, as mentioned in Remark~\ref{rem:open}, it follows that the product $U := \prod_i U_i$ of the $U_i$ in $\LRS$ is an open neighborhood of $\ov{x}$ in $X$,\footnote{This is the only place we need ``finite".  If $\{ X_i \}$ were infinite, the topological space product of the $U_i$ might not be open in the topology on the topological space product of the $X_i$ because the product topology only allows ``restriction in \emph{finitely many} coordinates".} so it remains only to prove that there is an isomorphism $U \cong \Spec \otimes_i A_i$, hence $U$ is affine.\footnote{There would not be a problem \emph{here} even if $\{ X_i \}$ were infinite: $\An$ has all direct and inverse limits, so the (possibly infinite) tensor product $\otimes_i A_i$ over $\ZZ$ (coproduct in $\An$) makes sense.  Our proof therefore shows that any inverse limit (not necessarily finite) of \emph{affine} schemes, taken in $\LRS$, is a scheme.}  Indeed, we can see immediately from Proposition~\ref{prop:affine} that $U$ and $\Spec \otimes_i A_i$ represent the same functor on $\LRS$: \be \Hom_{\LRS}(Y,U) & = & \prod_i \Hom_{\LRS}(Y,U_i) \\ & = & \prod_i \Hom_{\An}(A_i, \Gamma(Y,\O_Y)) \\ & = & \Hom_{\An}(\otimes_i A_i, \Gamma(Y,\O_Y)) \\ & = & \Hom_{\LRS}(Y, \Spec \otimes_i A_i). \ee  

The case of equalizers is similar: Suppose $X$ is the $\LRS$ equalizer of morphisms $f,g : Y \rightrightarrows Z$ of schemes, and $x \in X$.  Let $y \in Y$ be the image of $x$ in $Y$, so $f(y)=g(y)=:z$.  Since $Y,Z$ are schemes, we can find  affine neighborhoods $V = \Spec B$ of $y$ in $Y$ and $W = \Spec A$ of $z$ in $Z$ so that $f,g$ take $V$ into $W$.  As before, it is clear that the equalizer $U$ of $f|V, g|V : V \rightrightarrows W$ in $\LRS$ is an open neighborhood of $x \in X$, and we prove exactly as above that $U$ is affine by showing that it is isomorphic to $\Spec$ of the coequalizer \be C & = & B / \langle \{ f^{\sharp}(a) - g^{\sharp}(a) : a \in A \} \rangle \ee of $f^{\sharp}, g^{\sharp} : A \rightrightarrows B$ in $\An$. \end{proof}

\begin{rem} \label{rem:inverselimits}  The basic results concerning the existence of inverse limits in $\LRS$ and their coincidence with inverse limits in $\Sch$ are, at least to some extent, ``folk theorems".  I do not claim originality here.  The construction of fibered products in $\LRS$ can perhaps be attributed to Hanno Becker \cite{HB}, and the fact that a cartesian diagram in $\Sch$ is also cartesian in $\LRS$ is implicit in the [EGA] Erratum mentioned above. \end{rem}

\begin{rem} \label{rem:topoi}  It is unclear to me whether the 2-category of locally ringed topoi has 2-fibered products, though Hakim seems to want such a fibered product in [Hak V.3.2.3]. \end{rem}

\subsection{Fibered products} \label{section:fiberedproducts} In this section, we will more closely examine the construction of fibered products in $\LRS$ and explain the relationship between fibered products in $\LRS$ and those in $\RS$.  By Theorem~\ref{thm:Schinverselimits}, the inclusion $\Sch \into \LRS$ preserves inverse limits, so these results will generalize the basic results comparing fibered products in $\Sch$ to those in $\RS$ (the proofs will become much more transparent as well).

\begin{defn} \label{defn:S} Suppose $(A,\m,k), (B_1,\m_1,k_1), (B_2,\m_2,k_2) \in \LAn$ and $f_i : A \to B_i$ are $\LAn$ morphisms, so $f_i^{-1}(\m_i) = \m$ for $i=1,2$.  Let $i_j : B_j \to B_1 \otimes_A B_2$ ($j=1,2$) be the natural maps.  Set \bne{S} S(A,B_1,B_2) & := & \{ \p \in \Spec (B_1 \otimes_A B_2) : i_1^{-1}(\p)=\m_1, \; i_2^{-1}(\p) = \m_2 \} . \ene  Note that the kernel $K$ of the natural surjection \be B_1 \otimes_A B_2 & \to & k_1 \otimes_k k_2 \\ b_1 \otimes b_2 & \mapsto & [b_1] \otimes [b_2] \ee is generated by the expressions $m_1 \otimes 1$ and $1 \otimes m_2$, where $m_i \in \m_i$, so $$ \Spec (k_1 \otimes_k k_2) \into \Spec (B_1 \otimes_A B_2) $$ is an isomorphism onto $S(A,B_1,B_2)$.  In particular, \be S(A,B_1,B_2) & = & \{ \p \in \Spec (B_1 \otimes_A B_2) : K \subseteq \p \} \ee is closed in $\Spec (B_1 \otimes_A B_2)$. \end{defn}

The subset $S(A,B_1,B_2)$ enjoys the following important property: Suppose $g_i : (B_i,\m_i) \to (C,\n)$, $i=1,2$, are $\LAn$ morphisms with $g_1f_1 = g_2 f_2$ and $h = (f_1,f_2) : B_1 \otimes_A B_2 \to C$ is the induced map.  Then $h^{-1}(\n) \in S(A,B_1,B_2)$.  Conversely, every $\p \in S(A,B_1,B_2)$ arises in this manner: take $C = (B_1 \otimes_A B_2)_{\p}$.

\noindent {\bf Setup:}  We will work with the following setup throughout this section.  Let $f_1 : X_1 \to Y$, $f_2 : X_2 \to Y$ be morphisms in $\LRS$.  From the universal property of fiber products we get a natural  ``comparison" map \be \eta : X_1 \times^{\LRS}_Y X_2 & \to & X_1 \times^{\RS}_Y X_2 . \ee  Let $\pi_i : X_1 \times_Y^{\RS} X_2 \to X_i$ ($i=1,2$) denote the projections and let $g := f_1 \pi_1 = f_2 \pi_2$.  Recall that the structure sheaf of $X_1 \times_Y^{\RS} X_2$ is $\pi_1^{-1} \O_{X_1} \otimes_{g^{-1} \O_Y} \pi_2^{-1} \O_{X_2}$.  In particular, the stalk of this structure sheaf at a point $x = (x_1,x_2) \in X_1 \times_Y^{\RS} X_2$ is $\O_{X_1,x_1} \otimes_{\O_{Y,y}} \O_{X_2,x_2}$, where $$y := g(x) = f_1(x_1) = f_2(x_2).$$  In this situation, we set \be S(x_1,x_2) & := & S(\O_{Y,y}, \O_{X_1,x_1}, \O_{X_2,x_2}) \ee to save notation.

\begin{thm} \label{thm:comparison} The comparison map $\eta$ is surjective on topological spaces.  More precisely, for any $x = (x_1,x_2) \in X_1 \times_Y^{\RS} X_2$, $\eta^{-1}(x)$ is in bijective correspondence with the set $S(x_1,x_2)$, and in fact, there is an $\LRS$ isomorphism \be \eta^{-1}(x) & := & X_1 \times_Y^{\LRS} X_2 \times_{X_1 \times_Y^{\RS} X_2} (x, \O_{X_1,x_1} \otimes_{\O_{Y,y}} \O_{X_2,x_2}) \\ & = & \Spec_{ \O_{X_1,x_1} \otimes_{\O_{Y,y}} \O_{X_2,x_2} } S(x_1,x_2). \ee  In particular, $\eta^{-1}(x)$ is isomorphic as a \emph{topological space} to $\Spec k(x_1) \otimes_{k(y)} k(x_2)$ (but not as a ringed space).  The stalk of $\eta$ at $z \in S(x_1,x_2)$ is identified with the localization map \be  \O_{X,x_1} \otimes_{\O_{Y,y}} \O_{X,x_2}  & \to & ( \O_{X,x_1} \otimes_{\O_{Y,y}} \O_{X,x_2} )_z. \ee  In particular, $\eta$ is a localization morphism (Definition~\ref{defn:localizationmorphism}).  \end{thm}

\begin{proof} We saw in \S\ref{section:inverselimits} that the comparison map $\eta$ is identified with the localization of $X_1 \times_Y^{\RS} X_2$ at the prime system $(x_1,x_2) \mapsto S(x_1,x_2)$, so these results follow from Proposition~\ref{prop:localization}. \end{proof}

\begin{rem} When $X_1,X_2,Y \in \Sch$, the first statement of Theorem~\ref{thm:comparison} is [EGA I.3.4.7]. \end{rem}

\begin{rem} The fact that $\eta$ is a localization morphism is often implicitly used in the theory of the cotangent complex. \end{rem}

\begin{defn} \label{defn:rational} Let $f : X \to Y$ be an $\LRS$ morphism.  A point $x \in X$ is called \emph{rational} over $Y$ (or ``over $y := f(x)$" or ``with respect to $f$") iff the map on residue fields $\ov{f}_x : k(y) \to k(x)$ is an isomorphism (equivalently: is surjective).   \end{defn}

\begin{cor} \label{cor:comparison1} Suppose $x_1 \in X_1$ is rational over $Y$ (i.e.\ with respect to $f_1 : X_1 \to Y$).  Then for any $x=(x_1,x_2) \in X_1 \times^{\RS}_Y X_2$, the fiber $\eta^{-1}(x)$ of the comparison map $\eta$ is punctual.  In particular, if \emph{every} point of $X_1$ is rational over $Y$, then $\eta$ is bijective.  \end{cor}

\begin{proof} Suppose $x_1 \in X_1$ is rational over $Y$.  Suppose $x = (x_1,x_2) \in X_1 \times^{\RS} X_2$.  Set $y := f_1(x_1)=f_2(x_2)$.  Since $x_1$ is rational, $k(y) \cong k(x_1)$, so $\Spec k(x_1) \otimes_{k(y)} k(x_2) \cong \Spec k(x_2)$ has a single element.  On the other hand, we saw in Definition~\ref{defn:S} that this set is in bijective correspondence with the set $$ S(x_1,x_2) \subseteq \Spec ( \O_{X_1,x_1} \otimes_{\O_{Y,y}} \O_{X_2,x_2} )$$ appearing in Theorem~\ref{thm:comparison}, so that same theorem says that $\eta^{-1}(x)$ consists of a single point.  \end{proof}

\begin{rem} Even if every $x_1 \in X_1$ is rational over $Y$, the comparison map $$\eta : X_1 \times_Y^{\LRS} X_2 \to X_1 \times_Y^{\RS} X_2$$ is not generally an isomorphism on topological spaces, even though it is bijective.  The topology on $X_1 \times_Y^{\LRS} X_2$ is generally much finer than the product topology.  In this situation, the set $S(x_1,x_2)$ always consists of a single element $z(x_1,x_2)$: namely, the maximal ideal of $\O_{X_1,x_1} \otimes_{\O_{Y,y}} \O_{X_2,x_2}$ given by the kernel of the natural surjection $$\O_{X_1,x_1} \otimes_{\O_{Y,y}} \O_{X_2,x_2} \to k(x_1) \otimes_{k(y)} k(x_2) = k(x_2).$$  If we identify $X_1 \times_Y^{\LRS} X_2$ and $X_1 \times_Y^{\RS} X_2$ as sets via $\eta$, then the ``finer" topology has basic open sets \be  \U(U_1 \times_Y U_2,s) := \{ (x_1,x_2) \in U_1 \times_Y U_2 : s_{(x_1,x_2)} \notin z(x_1,x_2) \} \ee as $U_1,U_2$ range over open subsets of $X_1,X_2$ and $s$ ranges over $$(\pi_1^{-1}\O_{X_1} \otimes_{g^{-1}\O_Y} \pi_2^{-1} \O_{X_2})(U_1 \times_Y U_2).$$  This set is not generally open in the product topology because the stalks of $$\pi_1^{-1}\O_{X_1} \otimes_{g^{-1}\O_Y} \pi_2^{-1} \O_{X_2}$$ are not generally local rings, so not being in $z(x_1,x_2)$ does not imply invertibility, hence is not generally an open condition on $(x_1,x_2)$. \end{rem}

\begin{rem} On the other hand, sometimes the topologies on $X_1$, $X_2$ are so fine that the sets $\U(U_1 \times_Y U_2,s)$ \emph{are} easily seen to be open in the product topology.  For example, suppose $k$ is a topological field.\footnote{I require all finite subsets of $k$ to be closed in the definition of ``topological field".}  Then one often works in the full subcategory $\C$ of locally ringed spaces over $k$ consisting of those $X \in \LRS / k$ satisfying the conditions: \begin{enumerate} \item \label{an1} Every point $x \in X$ is a $k$ point: the composition $k \to \O_{X,x} \to k(x) $ yields an isomorphism $k=k(x)$ for every $x \in X$.  \item  The structure sheaf $\O_{X}$ is \emph{continuous for the topology on} $k$ in the sense that, for every $(U,s) \in \Sec \O_{X}$, the function \be s( \slot ) : U & \to & k \\ x & \mapsto & s(x) \ee is a continuous function on $U$.  Here $s(x) \in k(x)$ denotes the image of the stalk $s_x \in \O_{X,x}$ in the residue field $k(x) = \O_{X,x} / \m_x$, and we make the identification $k=k(x)$ using \eqref{an1}. \end{enumerate}  One can show that fiber products in $\C$ are the same as those in $\LRS$ and that the forgetful functor $\C \to \Top$ preserves fibered products (even though $\C \to \RS$ may not).  Indeed, given $s \in (\pi_1^{-1}\O_{X_1} \otimes_{g^{-1}\O_Y} \pi_2^{-1} \O_{X_2})(U_1 \times_Y U_2)$, the set $\U(U_1 \times_Y U_2,s)$ is the preimage of $k^* \subseteq k$ under the map $s( \slot )$, and we can see that $s( \slot )$ is continuous as follows:  By viewing the sheaf theoretic tensor product as the sheafification of the presheaf tensor product we see that, for any point $(x_1,x_2) \in \U_1 \times_Y U_2$, we can find a neighborhood $V_1 \times_Y V_2$ of $(x_1,x_2)$ contained in $U_1 \times_Y U_2$ and sections $a_1, \dots, a_n \in \O_{X_1}(V_1)$, $b_1,\dots,b_n \in \O_{X_2}(V_2)$ such that the stalk $s_{x_1',x_2'}$ agrees with $\sum_i (a_i)_{x_1'} \otimes (b_i)_{x_2'}$ at each $(x_1',x_2') \in V_1 \times_Y V_2$.  In particular, the function $s( \slot )$ agrees with the function $$(x_1',x_2') \mapsto \sum_i a_i(x_1')b_i(x_2') \in k$$ on $V_1 \times_Y V_2$.  Since this latter function is continuous in the product topology on $V_1 \times_Y V_2$ (because each $a_i( \slot )$, $b_i( \slot)$ is continuous) and continuity is local, $s( \slot )$ is continuous.    \end{rem}

\begin{cor} \label{cor:comparison2} Suppose $(f_1)_{x_1} : \O_{Y, f(x_1)} \to \O_{X_1,x_1}$ is surjective for every $x_1 \in X_1$.  Then the comparison map $\eta$ is an isomorphism.  In particular, $\eta$ is an isomorphism under either of the following hypotheses:  \begin{enumerate} \item \label{immersion} $f_1$ is an immersion. \item \label{point} $f_1 : \Spec k(y) \to Y$ is the natural map associated to a point $y \in Y$. \end{enumerate} \end{cor}

\begin{proof}  It is equivalent to show that $X := X_1 \times_Y^{\RS} X_2$ is in $\LRS$ and the structure maps $\pi_i : X \to X_i$ are $\LRS$ morphisms.  Say $x=(x_1,x_2) \in X$ and let $y := f_1(x_1)=f_2(x_2)$.  By construction of $X$, we have a pushout diagram of rings $$ \xym@C+15pt{ \O_{Y,y} \ar[r]^-{(f_1)_{x_1}} \ar[d]_{(f_2)_{x_2}} & \O_{X_1,x_1} \ar[d]^{(\pi_1)_x} \\ \O_{X_2,x_2} \ar[r]^-{(\pi_2)_x} & \O_{X,x} } $$ hence it is clear from surjectivity of $(f_1)_{x_1}$ and locality of $(f_2)_{x_2}$ that $\O_{X,x}$ is local and $(\pi_1)_x, (\pi_2)_x$ are $\LAn$ morphisms. \end{proof}

\begin{cor} \label{cor:comparison3} Suppose $$ \xym{ X_1 \times_Y X_2 \ar[r]^-{\pi_2} \ar[d]_{\pi_1} & X_2 \ar[d]^{f_2} \\ X_1 \ar[r]^{f_1} & Y } $$ is a cartesian diagram in $\LRS$.  Then: \begin{enumerate} \item \label{rationaletafiber} If $z \in X_1 \times_Y X_2$ is rational over $Y$, then $\eta^{-1}(\eta(z)) = \{ z \}$. \item \label{rationalbasechange} Let $(x_1,x_2) \in X_1 \times^{\RS}_Y X_2$, and let $y := \pi_1(x_1) = \pi_2(x_2).$   Suppose $k(x_2)$ is isomorphic, as a field extesion of $k(y)$, to a subfield of $k(x_1)$.  Then there is a point $z \in  X_1 \times_Y^{\Sch} X_2$ rational over $X_1$ with $\pi_i(z)=x_i$, $i=1,2$. \end{enumerate} \end{cor}

\begin{proof} For \eqref{rationaletafiber}, set $(x_1,x_2) := \eta(z)$, $y := \pi_1(x_1) = \pi_2(x_2)$.  Then we have a commutative diagram $$ \xym{ k(z) & \ar[l]_-{\ov{\pi}_{2,z}} k(x_2) \\ k(x_1) \ar[u]^{\ov{\pi}_{1,z}} & \ar[l]_-{\ov{f}_{1, x_1}} k(y) \ar[u]_-{\ov{f}_{2,x_2}} } $$ of residue fields.  By hypothesis, the compositions $k(y) \to k(x_i) \to k(z)$ are isomorphisms for $i=1,2$, so it must be that every map in this diagram is an isomorphism, hence the diagram is a pushout.  On the other hand, according to the first statement of Theorem~\ref{thm:comparison}, $\eta^{-1}(\eta(z))$ is in bijective correspondence with $$\Spec (k(x_1) \otimes_{k(y)} k(x_2)) = \Spec k(z),$$ which is punctual.

For \eqref{rationalbasechange}, let $i : k(x_2) \into k(x_1)$ be the hypothesized morphism of field extensions of $k(y)$.  By the universal property of the $\LRS$ fibered product $X_1 \times_Y X_2$, the maps \be (x_2,i): \Spec k(x_1) & \to & X_2 \\ x_1 : \Spec k(x_1) & \to & X_1 \ee give rise to a map $$ g : \Spec k(x_1) \to X_1 \times_Y X_2. $$  Let $z \in X_1 \times_Y X_2$ be the point corresponding to this map.  Then we have a commutative diagram of residue fields $$ \xym{ k(x_1) \\ & k(z) \ar[lu] & k(x_2) \ar[l] \ar[llu]_i \\ & k(x_1) \ar@{=}[luu] \ar[u]_{\ov{\pi}_{1,z}} & k(y) \ar[l] \ar[u]  }$$ so $\ov{\pi}_{1,z} : k(x_1) \to k(z)$ must be an isomorphism. \end{proof}

\subsection{Spec functor}  \label{section:relativespec} Suppose $X \in \LRS$ and $f : \O_X \to A$ is an $\O_X$ algebra.  Then $f$ may be viewed as a morphism of ringed spaces $f : (X,A) \to (X,\O_X)=X$.  Give $X$ the local prime system $\M_X$ as usual and $(X,A)$ the inverse image prime system $f^*\M_X$ (Remark~\ref{rem:inverseimage}), so $f$ may be viewed as a $\PRS$ morphism $$f : (X,A,f^* \M_X) \to (X,\O_X, \M_X).$$  Explicitly: \be (f^* \M_X)_x & = & \{ \p \in A_x : f_x^{-1}(\p) = \m_x \subseteq \O_{X,x} \}. \ee  By Theorem~\ref{thm:localization}, there is a unique $\LRS$ morphism $$\ov{f} : (X,A,f^* \M_X)^{\loc} \to (X,\O_X,\M_X)^{\loc}=X$$ lifting $f$ to the localizations.  We call \be \Spec_X A & := & (X,A,f^* \M_X)^{\loc} \ee the \emph{spectrum} (relative to $X$) of $A$.  $\Spec_X$ defines a functor \be \Spec_X : (\O_X / \An(X))^{\rm op} & \to & \LRS / X . \ee  Note that $ \Spec_X \O_X  =  (X,\O_X,\M_X)^{\loc}  =   X$ by Theorem~\ref{thm:localization}.

Our functor $\Spec_X$ agrees with the usual one (c.f.\ [Har II.Ex.5.17]) on their common domain of definition:

\begin{lem} Let $f : X \to Y$ be an affine morphism of schemes.  Then $\Spec_X f_* \O_X$ (as defined above) is naturally isomorphic to $X$ in $\LRS/Y$. \end{lem}

\begin{proof} This is local on $Y$, so we can assume $Y = \Spec A$ is affine, and hence $X = \Spec B$ is also affine, and $f$ corresponds to a ring map $f^\sharp : A \to B$.  Then \be f_* \O_X & = & B^{\sim} \\ & = & \underline{B}_Y \otimes_{ \underline{A}_Y } \O_Y, \ee as $\O_Y$ algebras, and the squares in the diagram $$ \xym{ (Y,f_* \O_X, (f^\flat)^* \M_Y) \ar[r] \ar[d] & (Y,\O_Y,\M_Y)  \ar[d] \\ (Y,\underline{B}_Y, \underline{f^{\sharp}}_Y^* N ) \ar[r] \ar[d] & (Y,\underline{A}_Y, N) \ar[d] \\ (*,B,\Spec B) \ar[r] & (*,A,\Spec A) } $$ in $\PRS$ are cartesian in $\PRS$, where $N$ is the prime system on $(Y,\underline{A}_Y)$ given by $N_y = \{ y \}$ discussed in Lemma~\ref{lem:important}.  According to that lemma, the right vertical arrows become isomorphisms upon localizing, and according to Theorem~\ref{thm:inverselimits}, the diagram stays cartesian upon localizing, so the left vertical arrows also become isomorphisms upon localizing, hence \be \Spec_Y f_* \O_X & := & (Y, f_*\O_X, (f^{\flat})^* \M_Y)^{\loc} \\ & = & \Spec B \\ & = & X. \ee \end{proof}

\begin{rem} Hakim [Hak IV.1] defines a ``Spec functor" from ringed topoi to locally ringed topoi, but it is not the same as ours on the common domain of definition.  There is no meaningful situation in which Hakim's Spec functor agrees with the ``usual" one.  When $X$ ``is" a locally ringed space, Hakim's $\Spec X$ ``is" (up to replacing a locally ringed space with the corresponding locally ringed topos) our $(X,\T_X)^{\loc}$.  As mentioned in Remark~\ref{rem:localization}, Hakim's theory of localization is only developed for the terminal prime system, which can be a bit awkward at times.  For example, if $X$ is a locally ringed space at least one of whose local rings has positive Krull dimension, Hakim's sequence of spectra yields an infinite strictly descending sequence of $\RS$ morphisms $$ \cdots \to \Spec (\Spec X) \to \Spec X \to X.$$ \end{rem}

The next results show that $\Spec_X$ takes direct limits of $\O_X$ algebras to inverse limits in $\LRS$ and that $\Spec_X$ is compatible with changing the base $X$.

\begin{lem} The functor $\Spec_X$ preserves inverse limits. \end{lem}

\begin{proof} Let $i \mapsto (f_i : \O_X \to A_i)$ be a direct limit system in $\O_X / \An(X)$, with direct limit $f : \O_X \to A$, and structure maps $j_i : A_i \to A$.  We claim that $\Spec_X A = (X,A,f^* \M_X)^{\loc}$ is the inverse limit of $i  \mapsto  \Spec_X A_i = (X,A_i,f_i^*\M_X)^{\loc}$.  By Theorem~\ref{thm:inverselimits}, it is enough to show that $(X,A,f^* \M_X)$ is the inverse limit of $i \mapsto (X,A_i,f_i^* \M_X)$ in $\PRS$.  Certainly $(X,A)$ is the inverse limit of $i \mapsto (X,A_i)$ in $\RS$, so we just need to show that $f^* \M_X = \cap_i j_i^*( f_i^* \M_X)$ as prime systems on $(X,A)$ (see the proof of Theorem~\ref{thm:inverselimits}), and this is clear because $j_i f_i = f$, so, in fact, $j_i^*(f_i^* \M_X) = f^* \M_X$ for every $i$. \end{proof}

\begin{lem} Let $f : X \to Y$ be a morphism of locally ringed spaces.  Then for any $\O_Y$ algebra $g : \O_Y \to A$, the diagram $$ \xym{ \Spec_X f^*A \ar[r] \ar[d] & \Spec_Y A \ar[d] \\ X \ar[r] & Y } $$ is cartesian in $\LRS$. \end{lem}

\begin{proof}  Note $f^*A := f^{-1}A \otimes_{f^{-1} \O_Y} \O_X$ as usual.  One sees easily that $$ \xym{ (X, f^*A, (f^{-1}g)^* \M_X) \ar[r] \ar[d] & (Y,A,g^* \M_Y) \ar[d] \\ (X,\O_X,\M_X) \ar[r] & (Y,\O_Y,\M_Y) } $$ is cartesian in $\PRS$ so the result follows from Theorem~\ref{thm:inverselimits}. \end{proof}

\begin{example} \label{example:Spec} When $X$ is a scheme, but $A$ is not a coherent $\O_X$ module, $\Spec_X A$ may not be a scheme.  For example, let $B$ be a local ring, $X := \Spec B$, and let $x$ be the unique closed point of $X$.  Let $A := x_*B \in \An(X)$ be the skyscraper sheaf $B$ supported at $x$.  Note $\O_{X,x} = B$ and \be  \Hom_{\An(X)}(\O_X,x_*B) & = & \Hom_{\An}(\O_{X,x}, B), \ee so we have a natural map $\O_X \to A$ in $\An(X)$ whose stalk at $x$ is $\Id : B \to B$.  Then $\Spec_X A = (\{ x \}, A)$ is the punctual space with ``sheaf" of rings $A$, mapping in $\LRS$ to $X$ in the obvious manner.  But $(\{ x \}, A)$ is not a scheme unless $A$ is zero dimensional. 

Here is another related pathology example:  Proceed as above, assuming $B$ is a local \emph{domain} which is not a field and let $K$ be its fraction field.  Let $A := x_* K$, and let $\O_X \to A$ be the unique map whose stalk at $x$ is $B \to K$.  Then $\Spec_X A$ is empty. \end{example}

Suppose $X$ is a scheme, and $A$ is an $\O_X$ algebra such that $\Spec_X A$ is a scheme.  I do not know whether this implies that the structure morphism $\Spec_X A \to X$ is an \emph{affine} morphism of schemes.

\subsection{Relative schemes}  We begin by recalling some definitions.

\begin{defn} \label{defn:fibered} ([SGA1], [Vis 3.1]) Let $F : \C \to \D$ be a functor.  A $\C$ morphism $f : c \to c'$ is called \emph{cartesian} (relative to $F$) iff, for any $\C$ morphism $g : c'' \to c'$ and any $\D$ morphism $h : Fc' \to Fc''$ with $Fg \circ h = Ff$ there is a unique $\C$ morphism $\ov{h} : c \to c''$ with $F \ov{h} = h$ and $f = g \ov{h}$.  The functor $F$ is called a \emph{fibered category} iff, for any $\D$ morphism $f : d \to d'$ and any object $c'$ of $\C$ with $Fc'=d'$, there is a cartesian morphism $\ov{f} : c \to c'$ with $F \ov{f}=f$.  A \emph{morphism} of fibered categories $$(F : \C \to \D) \to (F' : \C' \to D')$$ is a functor $G : \C \to \C'$ satisfying $F'G=F$ and taking cartesian arrows to cartesian arrows.  If $\D$ has a topology (i.e.\ is a site), then a fibered category $F : \C \to \D$ is called a \emph{stack} iff, for any object $d \in D$ and any cover $\{ d_i \to d \}$ of $d$ in $\D$, the category $F^{-1}(d)$ is equivalent to the category $F(\{ d_i \to d \})$ of \emph{descent data} (see [Vis 4.1]). \end{defn}

Every fibered category $F$ admits a morphism of fibered categories, called the \emph{associated stack}, to a stack universal among such morphisms [Gir I.4.1.2].
 
\begin{defn} \label{defn:SchX}  ([Hak V.1]) Let $X$ be a ringed space.  Define a category $\Sch_X^{\rm pre}$ as follows.  Objects of $\Sch_X^{\rm pre}$ are pairs $(U,X_U)$ consisting of an open subset $U \subseteq X$ and a scheme $X_U$ over $\Spec \O_X(U)$.  A morphism $(U,X_U) \to (V,X_V)$ is a pair $(U \to V, X_U \to X_V)$ consisting of an $\Ouv(X)$ morphism $U \to V$ (i.e. $U \subseteq V$) and a morphism of schemes $X_U \to X_V$ making the diagram \bne{mo} & \xym{ X_U \ar[d] \ar[r] & X_V \ar[d] \\ \Spec \O_X(U) \ar[r] & \Spec \O_X(V) } \ene commute in $\Sch$.  The forgetful functor $\Sch_X^{\rm pre} \to \Ouv(X)$ is clearly a fibered category, where a cartesian arrow is a $\Sch_X^{\rm pre}$ morphism $(U \to V, X_U \to X_V)$ making \eqref{mo} \emph{cartesian} in $\Sch$ (equivalently in $\LRS$).  Since $\Ouv(X)$ has a topology, we can form the associated stack $\Sch_X$.  The category of \emph{relative schemes over} $X$ is, by definition, the fiber category $\Sch_X(X)$ of $\Sch_X$ over the terminal object $X$ of $\Ouv(X)$. \end{defn}

(The definition of relative scheme makes sense for a ringed topos $X$ with trivial modifications.)

\subsection{Geometric realization} \label{section:geometricrealization} Now let $X$ be a \emph{locally} ringed space.  Following [Hak V.3], we now define a functor \be F_X : \Sch_X(X) & \to & \LRS / X \ee called the \emph{geometric realization}.  Although a bit abstract, the fastest way to proceed is as follows:

\begin{defn} \label{defn:LRSX} Let $\LRS_X$ be the category whose objects are pairs $(U,X_U)$ consisting of an open subset $U \subseteq X$ and a locally ringed space $X_U$ over $(U,\O_X|U)$, and where a morphism $(U,X_U) \to (V,X_V)$ is a pair $(U \to V, X_U \to X_V)$ consisting of an $\Ouv(X)$ morphism $U \to V$ (i.e. $U \subseteq V$) and an $\LRS$ morphism $X_U \to X_V$ making the diagram \bne{mor} & \xym{ X_U \ar[d] \ar[r] & X_V \ar[d] \\ (U,\O_X|U) \ar@{^(->}[r] & (V,\O_X|V) } \ene commute in $\LRS$.  The forgetful functor $(U,X_U) \mapsto U$ makes $\LRS_X$ a fibered category over $\Ouv(X)$ where a cartesian arrow is a morphism $(U \to V, X_U \to X_V)$ making \eqref{mor} cartesian in $\LRS$.  \end{defn}

In fact the fibered category $\LRS_X \to \Ouv(X)$ is a stack: one can define locally ringed spaces and morphisms thereof over open subsets of $X$ locally.  Using the universal property of stackification, we define $F_X$ to be the morphism of stacks (really, the corresponding morphism on fiber categories over the terminal object $X \in \Ouv(X)$) associated to the morphism of fibered categories \be F_X^{\rm pre} : \Sch_X^{\rm pre} & \to & \LRS_X \\ (U,X_U) & \mapsto & (U, X_U \times_{\Spec \O_X(U)}^{\LRS} (U,\O_X|U)) . \ee  The map $(U,\O_X|U) \to \Spec \O_X(U)$ is the adjunction morphism for the adjoint functors of Proposition~\ref{prop:affine}.  This functor clearly takes cartesian arrows to cartesian arrows.

\begin{rem} Although we loosely follow [Hak V.3.2] in our construction of the geometric realization, our geometric realization functor differs from Hakim's on their common domain of definition. \end{rem}

\subsection{Relatively affine morphisms}  Let $f : X \to Y$ be an $\LRS$ morphism.  Consider the following conditions: \begin{enumerate}[label=RA\theenumi., ref=RA\theenumi]  \item \label{RA11} Locally on $Y$ there is an $\O_Y(Y)$ algebra $A$ and a cartesian diagram $$ \xym{ X \ar[r] \ar[d]_f & \Spec A \ar[d] \\ Y \ar[r] & \Spec \O_Y(Y) } $$ in $\LRS$. \item \label{RA} There is an $\O_Y$ algebra $A$ so that $f$ is isomorphic to $\Spec_Y A$ in $\LRS/Y$.  \item \label{RASpecQCo} Same condition as above, but $A$ is required to be quasi-coherent. \item \label{RA4} For any $g : Z \to Y$ in $\LRS / Y$, the map \be \Hom_{\LRS / Y}(Z,X) & \to & \Hom_{\O_Y / \An(Y)}( f_* \O_X, g_* \O_Z) \\ h & \mapsto & g_* h^{\flat} \ee is bijective. \end{enumerate}

\begin{rem} The condition \eqref{RA11} is equivalent to both of the following conditions: \begin{enumerate}[label=RA1.\theenumi., ref=RA1\theenumi] \item \label{RA12} Locally on $Y$ there is a ring homomorphism $A \to B$ and a cartesian diagram $$ \xym{ X \ar[r] \ar[d]_f & \Spec B \ar[d] \\ Y \ar[r] & \Spec A } $$ in $\LRS$.  \item \label{RA13} Locally on $Y$ there is an affine morphism of schemes $X' \to Y'$ and a cartesian diagram $$ \xym{ X \ar[r] \ar[d]_f & X' \ar[d] \\ Y \ar[r] & Y' } $$ in $\LRS$. \end{enumerate}  The above two conditions are equivalent by definition of an affine morphism of schemes, and one sees the equivalence of \eqref{RA11} and \eqref{RA12} using Proposition~\ref{prop:affine}, which ensures that the map $Y \to \Spec A$ in \eqref{RA11} factors through $Y \to \Spec \O_Y(Y)$, hence \be X & = & Y \times_{\Spec A} \Spec B \\ & = & Y \times_{\Spec \O_Y(Y)} \Spec \O_{Y}(Y) \times_{\Spec A} \Spec B \\ & = & Y \times_{\Spec \O_Y(Y)} \Spec (\O_Y(Y) \otimes_A B).\ee \end{rem}

Each of these conditions has some claim to be the definition of a \emph{relatively affine morphism} in $\LRS$.  With the exception of \eqref{RA}, all of the conditions are equivalent, when $Y$ is a scheme, to $f$ being an affine morphism of schemes in the usual sense.  With the exception of \eqref{RA4}, each condition is closed under base change.  For each possible definition of a relatively affine morphism in $\LRS$, one has a corresponding definition of \emph{relatively schematic morphism}, namely: $f : X \to Y$ in $\LRS$ is relatively schematic iff, locally on $X$, $f$ is relatively affine.  

The notion of ``relatively schematic morphism" obtained from \eqref{RA11} is equivalent to: $f : X \to Y$ is in the essential image of the geometric realization functor $F_Y$.

\subsection{Monoidal spaces}  The setup of localization of ringed spaces works equally well in other settings; for example in the category of \emph{monoidal spaces}.  We will sketch the relevant definitions and results.  For our purposes, a \emph{monoid} is a set $P$ equipped with a commutative, associative binary operation $+$ such that there is an element $0 \in P$ with $0+p=p$ for all $p \in P$.  A morphism of monoids is a map of sets that respects $+$ and takes $0$ to $0$.  An \emph{ideal} of a monoid $P$ is a subset $I \subseteq P$ such that $I+P \subseteq I$.  An ideal $I$ is \emph{prime} iff its complement is a submonoid (in particular, its complement must be non-empty).  A submonoid whose complement is an ideal, necessarily prime, is called a \emph{face}.  For example, the faces of $\NN^2$ are $\{ (0,0) \}$, $\NN \oplus 0$, and $0 \oplus \NN$; the diagonal $\Delta : \NN \into \NN^2$ is a submonoid, but not a face.   

If $S \subseteq P$ is a submonoid, the \emph{localization} of $P$ at $S$ is the monoid $S^{-1}P$ whose elements are equivalence classes $[p,s]$, $p \in P$, $s \in S$ where $[p,s]=[p',s']$ iff there is some $t \in S$ with $t+p+s'=t+p'+s$, and where $[p,s]+[p',s']:=[p+p',s+s']$.  The natural map $P \to S^{-1}P$ given by $p \mapsto [p,0]$ is initial among monoid homomorphisms $h : P \to Q$ with $h(S) \subseteq Q^*$.  The localization of a monoid at a prime ideal is, by definition, the localization at the complementary face.

A \emph{monoidal space} $(X,\M_X)$ is a topological space $X$ equipped with a sheaf of monoids $\M_X$.  Monoidal spaces form a category $\MS$ where a morphism $f=(f, f^\dagger) : (X,\M_X) \to (Y,\M_Y)$ consists of a continuous map $f : X \to Y$ together with a map $f^\dagger : f^{-1} \M_Y \to \M_X$ of sheaves of monoids on $X$. A monoidal space $(X,\M_X)$ is called \emph{local} iff each stalk monoid $\M_{X,x}$ has a unique maximal ideal $\m_x$.  Local monoidal spaces form a category $\LMS$ where a morphism is a map of the underlying monoidal spaces such that each stalk map $f^\dagger_x : \M_{Y,f(x)} \to \M_{X,x}$ is \emph{local} in the sense $(f^\dagger)^{-1} \m_{f(x)} = \m_x$.  A \emph{primed monoidal space} is a monoidal space equipped with a set of primes $M_x$ in each stalk monoid $\M_{X,x}$.  The \emph{localization} of a primed monoidal space is a map of monoidal spaces $(X,\M_X,M)^{\rm loc} \to (X,\M_X)$ from a local monoidal space constructed in an obvious manner analogous to the construction of \S\ref{section:localization} and enjoying a similar universal property.  In particular, we let $\Spec P$ denote the localization of the punctual space with ``sheaf" of monoids $P$ at the terminal prime system.  A \emph{scheme over} $\FF_1$ is a locally monoidal space locally isomorphic to $\Spec P$ for various monoids $P$.  (This is not my terminology.) 

The same ``general nonsense" arguments of this paper allow us to construct inverse limits of local monoidal spaces, to prove that a finite inverse limit of schemes over $\FF_1$, taken in local monoidal spaces, is again a scheme over $\FF_1$, to construct a relative $\Spec$ functor \be \Spec : (\M_X / \Mon(X))^{\rm op} & \to & \LMS / (X,\M_X) \ee for any $(X,\M_X) \in \LMS$ which preserves inverse limits, and to prove that the natural map $$\Hom_{\LMS}((X,\M_X),\Spec P) \to \Hom_{\Mon}(P,\M_X(X))$$ is bijective.

\end{document}